\documentclass[12pt,a4wide,oneside, reqno]{amsart}
\usepackage[T1]{fontenc}
\usepackage{amssymb,amsmath,color,textcomp,url,tikz,a4wide, 
mathtools}
\usepackage{url}
\usepackage{etoolbox}
\usepackage{mathrsfs,mathtools}
\usepackage{amsthm}
\usepackage{mdwlist}
\usepackage[english]{babel}
\usepackage[utf8]{inputenc}
\usepackage{tikz}
\usepackage{tkz-euclide}
\usepackage{soul}
\usepackage{fancyhdr}
\usepackage{mdwlist}
\usepackage{enumerate,xspace}
\RequirePackage{ifpdf}
\ifpdf
\usepackage[pdftex]{hyperref}
\else
\usepackage[hypertex]{hyperref}
\fi

\theoremstyle{plain}
\newtheorem{theorem}{Theorem}[section]
\newtheorem*{theorem2}{Main Theorem}
\newtheorem{cor}[theorem]{Corollary}
\newtheorem{prop}[theorem]{Proposition}
\newtheorem*{prop2}{Proposition}
\newtheorem{lemma}[theorem]{Lemma}

\theoremstyle{definition}
\newtheorem{remark}[theorem]{Remark}
\newtheorem{fact}[theorem]{Fact}
\newtheorem{definition}[theorem]{Definition}
\newtheorem{example}[theorem]{Example}

\newtheorem*{hyp}{Property ($\divideontimes$)}

\newcommand{\aq}{\bar{a}}
\newcommand{\bq}{\bar{b}}

\newcommand{\setN}{\mathbb{N}}

\newcommand{\lingua}{\mathcal{L}}
\newcommand{\M}{\mathcal{M}}
\newcommand{\N}{\mathcal{N}}
\newcommand{\eqL}{\overset{\lingua}{\equiv}}
\newcommand{\eqLn}{\overset{0}{\equiv}}

\def\Ind#1#2{#1\setbox0=\hbox{$#1x$}\kern\wd0\hbox to 0pt{\hss$#1\mid$\hss}
	\lower.9\ht0\hbox to 0pt{\hss$#1\smile$\hss}\kern\wd0}
\def\ind{\mathop{\mathpalette\Ind{}}}
\def\indS{\mathop{\mathpalette\Ind{}^{\star}}}
\def\indLd{\mathop{\mathpalette\Ind{}^{ld}}}
\def\indL{\mathop{\mathpalette\Ind{}^{\lingua}}}
\def\indLn{\mathop{\mathpalette\Ind{}^{0}}}
\def\notind#1#2{#1\setbox0=\hbox{$#1x$}\kern\wd0
	\hbox to 0pt{\mathchardef\nn=12854\hss$#1\nn$\kern1.4\wd0\hss}
	\hbox to 0pt{\hss$#1\mid$\hss}\lower.9\ht0 \hbox to 0pt{\hss$#1\smile$\hss}\kern\wd0}

\makeatletter
\def\namedlabel#1#2{\begingroup
	\def\@currentlabel{#2}%
	\phantomsection\label{#1}\endgroup
}
\makeatother

\DeclareMathOperator{\tp}{tp}

\DeclareMathOperator{\stp}{stp}

\DeclareMathOperator{\acl}{acl}
\DeclareMathOperator{\dcl}{dcl}
\DeclareMathOperator{\dc}{dc}
\DeclareMathOperator{\Stab}{Stab}
\DeclareMathOperator{\St}{St}

\begin{document}

	\title[Stationarity and elimination of imaginaries]{Stationarity and elimination of imaginaries in stable and simple theories}
	\date{\today}
	\author{Charlotte Bartnick}
	\address{Abteilung für Mathematische Logik; Mathematisches Insitut; Universität Freiburg; Ernst-Zermelo-Straße 1; 79104 Freiburg; Germany }
	\email{charlotte.bartnick@math.uni-freiburg.de}
	\keywords{Stability, Stationarity, Imaginaries, Pairs, Separably Closed Fields}
	\subjclass{2010: 03C45}
	
	\begin{abstract}
	We show that types over real algebraically closed sets are stationary, both for the theory of separably closed fields of infinite degree of imperfection and for the theory of beautiful pairs of algebraically closed field. The proof is given in a general setup without using specific features of theories of fields. 
	
	Moreover, we generalize results of Delon as well as of Messmer and Wood that separably closed fields of infinite degree of imperfection and differentially closed fields of positive characteristic do not have elimination of imaginaries. Using work of Wagner on subgroups of stable groups, we obtain a general criterion yielding the failure of geometric elimination of imaginaries. This criterion applies in particular to beautiful pairs of algebraically closed fields, giving an alternative proof of the corresponding result of Pillay and Vassiliev.
	\end{abstract}

	\maketitle

	\section*{Introduction}

	Fields with extra structure are an important object of study in model theory. Classical examples of such theories are the theory of differentially closed fields in characteristic $0$ (see \cite{Blum},\cite{McGrail}), the theory of algebraically closed fields with an automorphism \cite{ACFA}, the theory of beautiful pairs of algebraically closed fields in a fixed characteristic \cite{P83Pairs} or the theory of separably closed fields in a given characteristic and of a fixed degree of imperfection \cite{D88Ideaux}.
	
	Often, such theories can be treated in a uniform way (see e.g. \cite{BMP19Simple}) due to the fact that they are either obtained from an algebraically closed field by adding new symbols (e.g. a derivation, a predicate) or their models embed into an algebraically closed field (e.g. separably closed fields). Concepts such as non-forking independence or model theoretic algebraic closure can then be described in terms of the stable theory $T_0$ of algebraically closed fields. 
	
	In a broader context, not exclusively considering theories of fields, the present note follows the approach of \cite{BMP19Simple}:  Throughout the work we consider an $\lingua$-theory $T$ such that the models of $T$ embed into models of a stable $\lingua_0$-theory $T_0$ where $\lingua_0 \subset \lingua$ and $T_0$ has elimination of quantifiers and imaginaries. In Section \ref{SectionStationarity}, we first introduce this setting and then define properties of $T$ with respect to $T_0$ that allow to deduce stationarity of types over algebraically closed sets in the home sort (i.e. real algebraically closed sets) in $T$ from stationarity in $T_0$.  
	Our approach applies in particular to classical examples (see Section \ref{SectionExamples}), which yields the following results (see Theorem \ref{TheoremMain}, Corollary \ref{CorPairs} and Corollary \ref{CorSCF}):
	\begin{theorem2}
	Given a stable $\lingua$-theory $T$ controlled by the stable $\lingua_0$-theory $T_0$ (see Definition \ref{DefinitionControlled}), types in $T$ over real algebraically closed sets are stationary.
		
		In particular, types over real algebraically closed sets are stationary in the theory $SCF_p^\infty$ of separably closed fields of infinite degree of imperfection. The same is true for the theory $T_0P$ of beautiful pairs of a stable nfcp theory $T_0$ whenever $T_0$  has elimination of quantifiers 	
		as well as elimination of imaginaries.
	\end{theorem2}
	 We  present the application of the theorem to the examples mentioned above in Section \ref{SectionExamples}.	
	After proving the result, we realized that stationarity of types in separably closed fields already follows from work of Krupiński \cite{K06SCF}. Yet, we believe that the general theorem contains some  interesting ideas.

	We will  state without proof all the properties used for our examples in Section \ref{SectionExamples}. The reader is referred to \cite{D99Separably} for further explanation on the theory of separably closed fields and to \cite{P83Pairs} as well as \cite{BYPV03Pairs} for material on beautiful pairs.

	The second part of this note links the results to non-elimination of imaginaries. In stable theories, weak elimination of imaginaries implies stationarity of types over real algebraically closed sets.
	Delon proved in \cite{D88Ideaux} that the theory  $SCF_p^\infty$ does not have elimination of imaginaries. Moreover, Messmer and Wood showed the same result for $DCF_p$ in \cite{MW95der}. 
	By a result of Pillay and Vassiliev in \cite{PV04Imaginaries}, the theory of pairs does not have geometric elimination of imaginaries if an infinite group is definable in $T_0$. 
	
	Motivated by Delon's proof and by the proof of Messmer and Wood, we formulate two variants of a general criterion that yield the failure of geometric elimination of imaginaries in an arbitrary simple $\lingua$-theory $T$ with $T_0^\forall \subset T$ for a stable $\lingua_0$-theory $T_0$, even after naming the elements of a model of $T$.  More precisely, we assume the existence of two new $\lingua$-definable subgroups $H_2 \leq H_1$ of an $\lingua_0$-definable group $G$. 
	The $\lingua$-definable groups $H_1$ and $H_2$ can be considered in a model $\M_0$ of $T_0$ as (possibly non-definable) subgroups of the stable group $G$. Wagner's results on definable hulls in stable groups from \cite{W90Subgroups} 
	enable us to deduce  in Proposition \ref{PropEliminationCriterionStable} and \ref{PropEliminationCriterionSimple} that imaginaries representing certain cosets of $H_2$ in $H_1$ cannot be geometrically eliminated. 
	
We state below a customized version of Proposition \ref{PropEliminationCriterionSimple} with $H_1=G(M)$:
		\begin{prop2}
			Let $(G,\cdot)$ be an $\lingua_0$-definable group in the stable $\lingua_0$-theory $T_0$. 	Given a simple $\lingua$-theory $T$ with $T_0^\forall \subset T$ satisfying Property \ref{HypothesisDcl} (see Section \ref{SectionStationarity}), suppose that there is a model $\M$ of $T$ such that $(G(M),\cdot)$ is a group in $\M$.
			
			Assume further that there exists an $\lingua$-definable subgroup $H$ of infinite index in $G(M)$, defined over a a small elementary substructure $\N$ of $\M$, and that: 
		\begin{enumerate}[(1)]
			\item for every $\lingua$-generic $g$ in $G(M)$ over $N$ we have that $\acl_{\lingua}(g,N) \subset \acl_0(g,N)$, and
			\item the $\lingua_0$-definable hull $\dc_G(H(M))$ of $H(M)$ in $G$ equals $G$.
		\end{enumerate}
		Then, the theory $T$ does not have geometric elimination of imaginaries, even after adding constants for the model $\N$ to the language.
	\end{prop2}
	
	We assume a certain familiarity with stability theory and the general theory of stable groups (see for example \cite{P01Groups} or \cite{W97StableGroups}). In Section \ref{SectionHull} we will introduce the work of Wagner on definable hulls and the setting in which we use his results. The criterion itself is presented in Section \ref{SectionImaginaries} where we also discuss variants and equivalent formulations. At the end of the section, we show that the criterion applies to the theories mentioned above (beautiful pairs and separably closed fields) and to the new example of separably differentially closed fields (see \cite{IS23differentially}). We hope that the criterion may provide more examples in the future.

	\subsection*{Acknowledgments} Part of this work was done during a research stay in Lyon supported by the ANR GeoMod (ANR-DFG,
	AAPG2019). The author would like to thank Thomas Blossier for the helpful scientific discussions in Lyon and her supervisor Amador Martin-Pizarro for all the advice on writing this note. Moreover, the author would like to thank the anonymous referee whose suggestions have improved the presentation of this article.

		\section{The setting and stationarity}\label{SectionStationarity}

\textbf{We fix a complete stable theory $T_0$ in a language $\lingua_0$ and assume all throughout this article that $T_0$ has  elimination of quantifiers and elimination of imaginaries.} Note that we may always achieve quantifier elimination by a Morleyization and elimination of imaginaries by considering $T_0^{eq}$. We will work inside a sufficiently saturated model $\mathcal{M}_0$ of $T_0$. 
	
	We now consider a theory $T$ in a language $\lingua$ expanding $\lingua_0$ such that
		$T_0^\forall \subset T$, or equivalently,
		such that every model $\mathcal{M}$ of $T$ embeds (as a structure in the sublanguage $\lingua_0$) into the model $\mathcal{M}_0$ of $T_0$. Since $\lingua$ is an expansion of $\lingua_0$, quantifier elimination yields that the $\lingua_0$-type of a tuple $\bar{m}$ in $M$ (seen as a tuple in $\mathcal{M}_0$) is completely determined by its quantifier-free type inside the $\lingua_0$-substructure $\mathcal{M}$, and in particular by its $\lingua$-type in $\mathcal{M}$. It follows immediately that the $\lingua_0$-definable and algebraic closures of subsets of $\mathcal{M}$ do not depend on the embedding into a model of $T_0$.
	
	From now on, we consider a sufficiently saturated model $\mathcal{M}$ of the
	theory $T$ inside $\M_0$. If not explicitly stated otherwise, all tuples and subsets will be taken in $\mathcal{M}$. If we want to refer to their properties with respect to the theory $T_0$, we will use the index $0$ and otherwise the index $\lingua$. For example, we denote by $\dcl_{\lingua}(A)$ (resp. $\dcl_0(A)$) the definable closure in $T$ (resp. in $T_0$) of a subset $A$. 
	
	Quantifier elimination of $T_0$ and the above discussion yield the following easy observations:
	
	\begin{remark}\label{RemarkBasicFactsOnSetting}
		
		Consider a subset $A$ as well as tuples $\bar{b}$ and $\bar{b}'$. 
		
		\begin{enumerate}
			\item  If $\tp_{\lingua}(\bq/A)=\tp_\lingua(\bq'/A)$ (also denoted as $\bar{b} \eqL_A \bar{b}'$), then $\bar{b} \eqLn_A \bar{b}'$. \label{ItemTypLImpliziertTyp0}
			\item The $\lingua_0$-substructure $\langle A \rangle_0 $ generated by $A$ is contained in the $\lingua$-substructure $ \langle A \rangle_{\lingua}$ and hence in $ \mathcal{M}$. \label{ItemZusammenhangDerUS}
			\item If $A$ is $\lingua$-definably closed, then  $\dcl_0(A) \cap\mathcal{M} =A$. More generally we have that 
			$\dcl_0(A)\cap \mathcal{M} \subset \dcl_{\lingua}(A)$.  \label{ItemZusammenhangDerDcls} 
			\item \label{ItemZusammenhangDerAcls} Similarly, $\acl_0(A)\cap \mathcal{M} \subset \acl_{\lingua}(A)$, so  $\acl_0(A) \cap \mathcal{M} =A$ if $A$ is $\lingua$-algebraically closed.  
			\item  If $T_0$ is a reduct of $T$, we have that 
			$\acl_0(A) \subset \acl_{\lingua}(A)$, so $\lingua$-algebraically closed subsets are also $\lingua_0$-algebraically closed  \label{ItemReduktZusammenhangDerAcls}.
		\end{enumerate}
		
	\end{remark}

	In Section \ref{SectionExamples} we will show that types over real algebraically closed sets  in certain classical stable theories are stationary. The properties described here do not require $T$ to be stable and work all under the more general assumption of simplicity. However, the result of Theorem \ref{TheoremMain} will imply $T$ to be stable. For the sake of generality, we will from now on assume that $T$ is simple and denote by $\indL$ the non-forking independence in $T$, whilst $\indLn$ denotes non-forking independence in the sense of $T_0$.
	
	The following general property of $T$ with respect to $T_0$ will be used several times throughout this work.
	
		\begin{hyp}\namedlabel{HypothesisDcl}{($\divideontimes$)}(\cite[Hypothèse 1]{BMP19Simple})
		For every $\lingua$-elementary substructure $\mathcal{N}$ of $\M$, every element of  $\dcl_0({N})$ is $\lingua_0$-interdefinable with a tuple from $N$.
	\end{hyp}
	\begin{remark}\label{RemarkDClRedukt}\label{RemarkHypothesisDclTrue}
		In most of our examples, the theory $T$ will be a theory of fields and $T_0$ the theory of algebraically closed fields in characteristic $p$ (for $p$ prime or $0$). In this case, the definable closure $\dcl_0(N)$ of a model $\N$ of $T$ equals the field theoretic inseparable closure.  Hence, for any element in $\dcl_0({N})$, a certain $p$-power lies in ${N}$, so Property \ref{HypothesisDcl} holds.
		
		Moreover, if $T_0$ is a reduct of $T$, Property \ref{HypothesisDcl} is clearly satisfied.
	\end{remark}
	  Using Remark \ref{RemarkBasicFactsOnSetting} (\ref{ItemZusammenhangDerAcls}), {Blossier and Martin-Pizarro} showed the following:

	\begin{fact}(\cite[Remarque 1.7]{BMP19Simple})\label{FactStationarityT0}
		Given a simple $\lingua$-theory $T$ with $T_0^\forall \subset T$ and such that Property \ref{HypothesisDcl} holds, types in $T_0$ over $\lingua$-algebraically closed sets are stationary. 
	\end{fact}
	To prove this fact, stability of $T_0$ and elimination of imaginaries are crucial since they yield that types in $T_0$ over real $\lingua_0$-algebraically closed sets are stationary. In the particular case that $T_0\subset T$, every $\lingua$-algebraically closed set $A$ is also $\lingua_0$-algebraically closed (see Remark \ref{RemarkBasicFactsOnSetting} (\ref{ItemReduktZusammenhangDerAcls})) and thus the fact follows immediately. 
	\begin{remark}\label{RemarkStationarity}
		 The following condition (which was first considered by Lascar for models in \cite[Section 3.1]{L89amalgam}) is equivalent to the stationarity of the type $\tp_0(B/A)$:
		  
		 For every set $C \supset A$ and $\lingua_0$-elementary maps $f:B \to B'$ and $g: C \to C'$ such that ${f\hspace{-1mm}\upharpoonright}_A={g\hspace{-1mm}\upharpoonright}_A $ and $B \indLn_A C$ as well as $B' \indLn_{f(A)} C'$, the map $f\cup g: B \cup C \to B' \cup C'$ is $\lingua_0$-elementary.
	\end{remark}
	Blossier, Martin-Pizarro and Wagner showed in \cite[Lemme 2.1]{BMPW15Relative} for the case of $T_0$ being a reduct of $T$ that non-forking independence in $T$ implies independence in $T_0$. With our additional Property \ref{HypothesisDcl}, their proof can be easily adapted and yields the following:
	\begin{fact}(\cite[Lemme 2.1]{BMPW15Relative})\label{FactIndependenceImplication}
		Let $T$ be a simple $\lingua$-theory with $T_0^\forall \subset T$ and such that Property \ref{HypothesisDcl} is satisfied. For all  $A=\acl_{\lingua}(A)$ algebraically closed and $B$ and $C$ arbitrary subsets holds
		$$B \indL_A C \quad \Rightarrow \quad B \indLn_A C.$$			
	\end{fact}
	To state the other assumptions we need to fix a suitable class of substructures.
	\begin{definition}
		A class $\mathcal{K}$ of $\lingua$-substructures of $\M$ is called \emph{strong} if the following property holds: 		
		Every $\lingua$-isomorphism $f:A \to A'$ with $A$ and $A'$ in $\mathcal{K}$ is elementary (i.e. $A \eqL A'$).
		
		The elements of $\mathcal{K}$ are called \textit{strong} substructures. 
		
	\end{definition}
	In other words, any two substructures from $\mathcal{K}$ having the same quantifier free $\lingua$-type have the same $\lingua$-type. If $T$ has quantifier elimination, the class of all substructures is strong. We will discuss in Section \ref{SectionExamples} other examples of theories, without quantifier elimination, which admit a suitable class of strong substructures.
	\begin{remark}
		Suppose that $A$ is a strong substructure of $\M$ and $f:A \to \M$ an embedding with $f(A)$ also strong. Since $f$ is an elementary map, the set $A$ is $\lingua$-algebraically (or definably) closed if and only if the same holds for $f(A)$.
	\end{remark}
	\begin{definition}\label{DefinitionControlled}
		The simple $\lingua$-theory $T$ is \emph{controlled by} the stable $\lingua_0$-theory $T_0$ (which has elimination of imaginaries and quantifiers)\emph{ with respect to} the strong class $\mathcal{K}$ if the following conditions hold:
		\begin{enumerate}[(a)]
			\item We have $\lingua_0 \subset \lingua$ and $T_0^\forall \subset T$.
			\item The theory $T$ satisfies Property \ref{HypothesisDcl}.
			\item The class $\mathcal{K}$ contains all $\lingua$-definably closed sets.
			
			\item \label{PropertyUS}  For strong subsets $B$ and $C$ such that $A= B\cap C$ is $\lingua$-algebraically closed and $B \indL_A C$, (the universe of) the $\lingua_0$-generated substructure $\langle B, C \rangle_0$ coincides with the $\lingua$-generated substructure $\langle B, C \rangle_\lingua$ and is strong.
			\item \label{PropertyMorph}  For $A, B$ and $C$ as in (\ref{PropertyUS}), whenever ${f:\langle B , C \rangle_{\lingua} \to \M}$ is an $\lingua_0$-embedding whose restrictions ${f\hspace{-1mm}\upharpoonright}_B$ and ${f\hspace{-1mm}\upharpoonright}_C$ are $\lingua$-embeddings, if $f(B), f(C)$ and  $f(A)$ are strong with $f(B) \indL_{f(A)} f(C)$, then $f$ is an $\lingua$-embedding.	
		\end{enumerate}
	\end{definition}

	\begin{example}\label{ExampleRandomGraph}
		In the language $\lingua=\{R\}$ consisting of a binary relation symbol, consider the theory $T$ of the random graph  (with $\lingua_0$ the empty language and $T_0$ the theory of an infinite set). By quantifier elimination, the class of all substructures is strong and Property \ref{HypothesisDcl} holds trivially. Moreover,  (\ref{PropertyUS}) of Definition \ref{DefinitionControlled} is satisfied since $\lingua$ is purely relational.  		
				
		However, types over real algebraically closed sets in $T$ are not stationary: The independence $\indL$ is given by purely algebraic independence, but the relations between elements of $B$ and $C$ are not determined by the relations within $B$ and $C$. Indeed, by this reasoning (\ref{PropertyMorph})  of Definition \ref{DefinitionControlled} is not satisfied. 
		
	\end{example}

	\begin{theorem}\label{TheoremMain} 
		If the simple $\lingua$-theory $T$ is controlled by the stable $\lingua_0$-theory $T_0$ (see Definition \ref{DefinitionControlled}), then types over real algebraically closed sets are stationary.
	\end{theorem}
	In particular, types over models of $T$ are stationary and thus $T$ is stable.
	\begin{proof} Let $A=\acl_{\lingua}(A)$ be a real algebraically closed set and $\bq$ a tuple in $M$. In order to show that the type $\tp_\lingua(\bar{b}/A)$ is stationary,  choose $\bar{b}'$  such that $\bar{b}' \eqL_A \bar{b}$ and suppose that for some $C \supset A$ both $\bar{b} {\indL_A} C$ and $\bar{b}'\indL_A C$. We want to conclude $\bar{b}\eqL_C\bar{b}'$.
		
		We may assume that $C$ is $\lingua$-definably closed. Set $B=\dcl_{\lingua}(\bq A)$ and $B'=\dcl_{\lingua}(\bq' A)$. From the independence and the $\lingua$-algebraically closedness of $A$ we get \[B\cap C=A=B'\cap C.\]

		 By Fact \ref{FactIndependenceImplication}, we obtain from Property \ref{HypothesisDcl} the independences $B \indLn_A C $ and $B' \indLn_A C$. The assumption $B \eqL_A B'$ yields an $\lingua$-Isomorphism $g: B \to B'$ fixing $A$. This map $g$ is in particular an $\lingua_0$-Isomorphism and thus an elementary map by quantifier elimination of $T_0$.

		  As the type $\tp_0(B/A)$ is stationary by Fact \ref{FactStationarityT0} , we can amalgamate $g$ with $\mathrm{id}_C$ to an $\lingua_0$-elementary map $g \cup \mathrm{id}_C$ from $B \cup C$ to $B' \cup C$ (see Remark \ref{RemarkStationarity}). By elementarity, the map $g \cup \mathrm{id}_C$ extends to an $\lingua_0$-isomorphism  $f: \langle B , C \rangle_{0} \to \langle B' , C \rangle_{0}$ fixing $C$.
		 
		 Note that by (\ref{PropertyUS}) of Definition \ref{DefinitionControlled} (using (c)), the substructures $\langle B ,C \rangle_{0}$ and $\langle B',C \rangle_{0}$ coincide with the $\lingua$-generated ones. Hence, by  (\ref{PropertyMorph}) the map $f$ is an $\lingua$-embedding of $\langle B ,C \rangle_{\lingua}$ into $\M$ with image $\langle B',C \rangle_{\lingua}$. Since the $\lingua$-substructures $\langle B,C \rangle_{\lingua}$ and $\langle B',C \rangle_{\lingua}$ are in the controlling class $\mathcal{K}$ by  (\ref{PropertyUS}),  the embedding $f$ witnesses  that $$\langle B,C \rangle_{\lingua} \eqL \langle B',C \rangle_{\lingua}.$$
			In particular, $\bar{b}\eqL_C \bar{b}'$ as desired.
	\end{proof}
	
	Note that if (\ref{PropertyUS}) and (\ref{PropertyMorph}) of Definition \ref{DefinitionControlled} hold for a certain subclass of algebraically closed substructures $A$, the above proof yields stationarity of types over such sets.

	\section{Examples}\label{SectionExamples}

	In this section we will discuss how Theorem \ref{TheoremMain} applies to some classical theories.
	\subsection*{Beautiful pairs}

	\label{FactBP}
	Given a stable $\lingua_0$-theory $T_0$, Poizat considered a proper pair $\mathcal{M}' \prec \mathcal{M}$  of models of $T_0$, with ${M}'\subsetneq {M}$, in the language $\lingua=\lingua_0\cup \{P\}$. Such a pair is called \textit{$\kappa$-beautiful} (in the terminology of \cite{BYPV03Pairs}) if $\mathcal{M}'$ is $\kappa$-saturated and every $\lingua_0$-type over $BM'$, with $B$ a subset of $M$ of size less than $\kappa$, is realized in $\mathcal{M}$.

	Poizat proved in \cite{P83Pairs} that any two $\vert  T_0 \vert ^+$-beautiful pairs are elementary equivalent. Hence, the common theory $T_0P$ of such beautiful pairs is complete. If $T_0$ does not have the  finite cover property, every $\vert T_0 \vert^+$-saturated model is a beautiful pair. Moreover, the theory $T_0P$ is again stable and without the finite cover property. 
	
	We now assume that $T_0$ has elimination of quantifiers 	
	as well as elimination of imaginaries but not the finite cover property (nfcp). 
	We denote a sufficiently saturated model of $T$ by $(\M,P(M))$.
	
	\begin{fact}\cite[Lemma 3.8]{BYPV03Pairs}\label{FactStrongPairs}
		The class $\mathcal{K}$ of $\lingua$-substructures $A$ of $\M$ with  $A \indLn_{P(A)} P(M)$ is strong. (These strong substructures were called  $P$-independent in \cite{BYPV03Pairs}.)
	\end{fact}
	In \cite[Proposition 7.3]{BYPV03Pairs} Ben-Yaacov, Pillay and Vassiliev gave the following characterization of independence: For strong substructures $A, B$ and $C$ such that $A \subset B\cap C$ we have \begin{equation}\label{Independenceinpaircharaterization}
	B\indL_A C\quad \Leftrightarrow \quad B \indLn_{A \cup P(M) } {C \cup P(M)} \quad \text{ and }  \quad B \indLn_{{A}} {C} \tag{$\star$}.
	\end{equation}
	\begin{lemma}\label{LemmaPairsProperty2}
		The theory $T_0P$ of beautiful pairs is controlled by $T_0$ with respect to the class $\mathcal{K}$ introduced in Fact \ref{FactStrongPairs}. 
	\end{lemma}
	\begin{proof}
		We check all the conditions of Definition \ref{DefinitionControlled}. Part (a) is obvious and Property \ref{HypothesisDcl} holds for the theory $T_0P$ by Remark \ref{RemarkHypothesisDclTrue}. 	Moreover, definably closed subsets are strong by \cite[Remark 7.2 (i)]{BYPV03Pairs}.
		
		 Since $\lingua\setminus \lingua_0$ does not contain any new function symbols,  it suffices to show for (\ref{PropertyUS}) that $\langle B, C \rangle_{0}$ is strong whenever the strong substructures $B$ and $C$ are independent over their intersection. 	A verbatim adaptation of \cite[ Lemme 1.2]{BMP14Groupes} yields from (\ref{Independenceinpaircharaterization}) that 
		\[B \cup C \indLn_{P(B)\cup P(C)} P(M),\]	so  $\langle B, C \rangle_{0}$ is strong. Moreover, the above $\lingua_0$-independence implies that
		\begin{equation}\label{PairsPpartIndependence}
		P(\langle B, C \rangle_{0}) =P(M) \cap \langle B, C \rangle_{0}=  \acl_0(P(B),P(C)) \cap \langle B ,C \rangle_{0}, \tag{$\diamondsuit$}
		\end{equation}
		since the right to left inclusion holds by  $\lingua_0$-algebraically closedness of $P(M)$.

		In order to verify (\ref{PropertyMorph}), we consider strong substructures $B$ and $C$ with $B \indL_{B\cap C} B$ such that $f(B)$ and $f(C)$ are also strong with $f(B)\indL_{f(B)\cap f(C)} f(C)$ where $f$ is  an $\lingua_0$-embedding $f:\langle B, C \rangle_{0}=\langle B, C \rangle_{\lingua} \to \M$   as in the statement. To show that $f$ is an $\lingua$-embedding, it suffices to prove that $f$ respects the predicate $P$. This follows from (\ref{PairsPpartIndependence}) (applied to $B$ and $C$ as well as $f(B)$ and $f(C)$) since the restrictions ${f\hspace{-1mm}\upharpoonright}_B$ and ${f\hspace{-1mm}\upharpoonright}_C$ witness that $f(P(B))=P(f(B))$ and $f(P(C))=P(f(C))$. 
	\end{proof}
	Theorem \ref{TheoremMain} immediately yields: 
	\begin{cor} \label{CorPairs}
		In the theory $T_0P$ of beautiful pairs of $T_0$, types over real algebraically closed sets are stationary. \qed
	\end{cor}
	In the case of beautiful pairs, our proof of stationarity resonates with the ideas of \cite[Proposition 7.5 (c)]{BYPV03Pairs}.
	\subsection*{Separably closed fields of infinite degree of imperfection}
	Recall that a field $K$ is \textit{separably closed} if $K$ has no proper separable algebraic extensions. We will restrict our attention to fields of prime characteristic $p$. The linear dimension $[K:K^p]$ of $K$ over the subfield of $p$-powers is always of the form $p^e$ for some natural number $e$ or infinite. The number $e$ resp. $\infty$ is called the \textit{degree of imperfection}. Ershov showed in \cite{E87Fields} that the class of separably closed fields of characteristic $p$ and fixed degree of imperfection, axiomatized in the ring language, yields a complete stable theory. 
	
	In this subsection, we consider the theory $T=SCF_p^\infty$ of separably closed fields of infinite degree of imperfection  in the language $\lingua=\{0,1,+,-,\cdot, ^{-1}\}$. Note that the function symbol $^{-1}$ in the language ensures that every substructure is a subfield. The theory contains the universal part of  the theory $T_0=ACF_p$ of algebraically closed fields in the language $\lingua_0=\lingua$.	
	Consider a sufficiently saturated model $\M$ of $SCF_p^\infty$.
	
	Recall that an extension $L/K$ of fields is \textit{separable} if $K$ is linearly disjoint from the subfield of $p$-powers $L^p$ over $K^p$, denoted by $K \indLd_{K^p}{L}^p$.

	\begin{remark}\label{RemarkDefinitonGoodSCF}
		The class $\mathcal{K}$ of $\lingua$-substructures $A$ of $\M$ such that the field extension $M/A$ is separable is strong. 
		
		 Indeed, the proof of completeness by Ershov in \cite{E87Fields} yields a description of types: Given $K_1$ and $K_2$ separably closed fields both of characteristic $p$ and the same degree of imperfection and $k_i \subset K_i$ subfields such that the extension $K_i/k_i$ is separable, then $k_1$ and $k_2$ have the same type if they are isomorphic as fields (c.f. \cite[Theorem p.717]{S86Independence}).

	\end{remark}
	
	The non-forking independence in separably closed fields was studied by Srour. His result  \cite[Theorem 13]{S86Independence} implies the following characterization for $SCF_p^\infty$: Given $k_1$ and $k_2$  strong subfields such that $k_0 = k_1 \cap k_2$ is relatively algebraically closed in $M$,   
	\begin{equation}\label{IndependenceCharSCF}
	k_1 \indL_{k_0} k_2 \quad \Leftrightarrow  \quad k_1 \indLn_{k_0} k_2 \text{ and $M/\langle k_1, k_2\rangle_{0} $ is separable.} \tag{$\star\star$}
	\end{equation}

	\begin{remark}\label{RemarkSCFProperty2}
		The theory $T=SCF_p^\infty$ is controlled by $T_0=ACF_p$ with respect to $\mathcal{K}$.
		
		Indeed, all Conditions of Definition \ref{DefinitionControlled} are satisfied: Property \ref{HypothesisDcl} holds by Remark \ref{RemarkHypothesisDclTrue} and definably closed sets are strong by \cite[Lemma 0]{S86Independence}. 
		Since $\lingua=\lingua_0$, (\ref{PropertyUS}) follows directly from the characterization of independence in (\ref{IndependenceCharSCF}). Part (\ref{PropertyMorph}) is again immediate as $\lingua=\lingua_0$. 			
	\end{remark}
	
	\begin{cor} \label{CorSCF}
		In the theory $SCF_p^\infty$ of separably closed fields of infinite degree of imperfection, types over real algebraically closed sets are stationary.\qed
	\end{cor}
	
	In \cite[Proposition 2.2.2.]{K06SCF} Krupiński showed that the theory $SCF_p^\infty$ satisfies a weak version of elimination of imaginaries. It already follows implicitly from his result (using \cite[Proposition 3.9]{CF04Imaginaries}) that types over real algebraically closed subsets in $SCF_p^\infty$ are stationary. Our approach however yields a more direct proof.
	
	\begin{remark}
		There are two more classical examples to which our results apply:
		The theory $DCF_0$ of differentially closed fields in characteristic $0$ in the language $\lingua=\lingua_0 \cup \{D\}$ and the theory $SCF_p^e$ of separably closed fields of characteristic $p$ and finite degree of imperfection $e$ in the language $\lingua=\lingua_0 \cup\{(c_i)_{i<e}, (\lambda_r)_{r<p^e}\}$ with constants for a $p$-basis.
		In both cases, the class of all substructures is strong by quantifier elimination. The other properties are easily verified.
		
		We do not claim that the stationarity of types over algebraically closed sets is new in these cases, since both theories have elimination of imaginaries.
	\end{remark}

 		\section{Definable hulls in stable groups}\label{SectionHull}
 		\subsection{General results}\label{SectionHullGeneralResults}
 		 In this subsection, we will present some results of Wagner on groups and their subgroups in stable theories. We work in a sufficiently saturated model of a stable theory (with elimination of imaginaries). The results presented below remain valid for type-definable stable groups but for the sake of the presentation, we will assume that ambient model is the universe of a group $(G,\cdot)$. Moreover, we identify a formula with the corresponding definable subset.

 	 Recall that a definable subset $\varphi(x,\aq)$ is \textit{generic} in the stable group $G$ if $G$ is covered by finitely many bi-translates of the form $g^{-1} \cdot \varphi(x,\aq)\cdot h^{-1}$ with $g$ and $h$ elements of $G$. Since the order property does not hold in stable theories, we always have that either a definable subset or its complement is generic in $G$. This yields the existence of \textit{generic} types of $G$, i.e. types that only contain generic formulae. Note that the	type $p=\tp(g/A)$ is generic if ($\circ$) for every $h$ in $G$  independent from $g$ the product $g\cdot h$ is independent from $h$. 
 		
 		The \textit{connected component} of $G$, denoted by $G^0$, is the intersection of all
 		definable subgroups of finite index. The (left-)\textit{stabilizer} of a stationary type $p=\tp(g/A)$ is the set \[\Stab_G(p)=\{h \in G \mid h \cdot \mathfrak{p}=\mathfrak{p}\}\] 
 		where $\mathfrak{p}$ is the unique global non-forking extension of $p$ and $g\cdot \mathfrak{p}$ is given by the natural action of $G$ on the space of types induced by left translation. Using the definability of types in stable theories, the stabilizer of $p$ can be written as the intersection of the definable subgroups \[\Stab_G(p,\varphi)=\{g \in G \mid \, \models  \forall \bar{z} (d_p\varphi(gx,\bar{z}) \leftrightarrow d_p\varphi(x,\bar{z}) )\}.\] A type is generic in $G$ if and only if its stabilizer equals $G^0$. We refer the reader to Poizat's volume on stable groups  \cite{P01Groups} for proofs and details.
 		\begin{remark}
 			In Section \ref{SectionImaginaries} we will also consider generic types and stabilizers for groups definable in simple theories. In this case, there is no notion of generic formula but a type is called generic if it satisfies the condition in $(\circ)$. Since types are not necessarily stationary, the left-stabilizer $\Stab_G(p)$ of a type $p$ is defined as the subgroup generated by the set \[\St(p)= \{h \in G \mid \exists g \models p \text{ with }g \ind_A h \text{ such that } h\cdot g \models p \}.\] Right-stabilizers are defined analogously. More details and proofs can be found in \cite{C26Groups}.
 		\end{remark}
 		 Wagner showed in \cite{W90Subgroups} that the notions for stable groups extend to non-definable subgroups of $G$.

 		\begin{definition}(\cite[Definitions 1.1.2, 2.1.1, 2.1.2 and 2.1.3]{W97StableGroups}) 	\label{DefinitionWagnerGenerics}
 			Let $H$ be a (possibly non-definable) subgroup of the stable group $(G, \cdot)$. 
 			\begin{enumerate}
 				\item A definable subset $\varphi(x,\aq)$ of $G$ is called \emph{generic for $H$} if $H$ is covered by finitely many bitranslates of $\varphi(x,\aq)$ by elements from $H$; that is, there are elements $g_1, \ldots ,g_n$ and $ h_1, \ldots,h_n$ of $H$ such that for every $h$ in $H$ holds $ \bigvee_{i=1}^n \varphi(g_i \cdot  h\cdot h_i,\aq )$.
 				\item 	A type $p$ in $ S_1(A)$ with $A \supset H$ is \emph{generic for $H$} if all its formulae are generic for $H$.
 				\item The \textit{definable hull} of $H$ (in $G$), denoted $\dc_G(H)$, is the intersection of all definable subgroups (with parameters)  of $G$ containing $H$. 
 			\end{enumerate}

 		\end{definition}	
 		As in the case of stable groups, generic types for $H$ always exist and a definable set is generic for $H$ if and only if $H$ is covered by left translates.
 		
 		Consider the subspace $S_1^H(A)$ which consists of the types in $S_1(A)$ that are finitely satisfiable in $H$. By coheirness, each type in $S_1^H(A)$ does not fork over $H$ and its restriction to $H$ is stationary. Every generic type for $H$ lies in $S_1^H(A)$, since we considered translates by elements from $H$ in the definition of genericity for $H$ in $G$.
 	
 		\begin{fact}(\cite[Lemma 25, Lemma 26, Remark 27 and Lemma 28]{W90Subgroups}) \label{FactWagnerResults}
 			\begin{enumerate}
 				\item \label{FactWagnerProprertiesDC} The groups $\dc_G(H)$ and $\dc_G(H)^0$ are type-definable over $H$, and hence do not depend on the ambient model.	
 		
 				\item \label{FactWagnerGenericsResult}A type $p$ in $ S_1^H(A)$ is generic for $H$ if and only if  $\Stab_G(p)=\dc_G(H)^0$.
 				
 				\noindent In particular, every type that is generic for $H$ in $G$ is also generic in the type-definable group $\dc_G(H)$.
 			\end{enumerate}
 		\end{fact}
 		The first part follows from the Baldwin-Saxl condition for definable groups in stable theories. For the second part, Wagner uses that a definable subgroup $H'$ contains $\dc_G(H)^0$ if and only if the index $[H:H\cap H']$ is finite. The general structure of proofs is as in the definable case.

 		Before exploring how Wagner's results can be used in the setting of the criterion given in the next section, we first introduce the following notion:
 		\begin{definition}\label{DefinitionPrincipal}
 			Consider two stable groups $(G_2, \cdot ) \subset (G_1,\cdot )$.  We say that $G_2$ is \emph{principal in} $G_1$ if they have the same connected components, i.e. $G_1^0=G_2^0$.
 		
 		\end{definition}
		In his lecture notes \cite{C26Groups}, Casanovas refers to this notion as being generic since it is equivalent to the partial type defining $G_2$ being generic in $G_1$ by  part (ii) of the next remark.  
 		\begin{remark}\label{RemarkVariantsGenericity}
 		Let $(G_2, \cdot ) \subset (G_1,\cdot )$ be two 
 		stable groups, both defined over $A$. The following conditions are equivalent:
 			\begin{enumerate}[(i)]
 				\item The group $G_2$ is principal in $G_1$.
 				\item Every generic type of $G_2$ (over $A$) is also a generic type of $G_1$.
 				\item There is some element $g$ of $G_2$ such that $\tp(g/A)$ is generic in $G_1$.
 			\end{enumerate}
 		\end{remark}
 	The equivalences follow from basic facts about connected components. We will give a short proof for the sake of completeness.
 	\begin{proof}
 		
 	$(i) \Rightarrow (ii):$ Take a generic type $p$ of $G_2$, so its stabilizer $\Stab_{G_2}(p)$ equals $G_2^0=G_1^0$ by assumption. Since $G_1$ is a supergroup of $G_2$, we have $\Stab_{G_1}(p) \supset \Stab_{G_2}(p)=G_1^0$ which implies that $p$ is generic in $G_1$.
 	
 	$(ii) \Rightarrow (iii)$: This implication is trivial.
 	
 	$(iii) \Rightarrow (i)$: We may assume that $A=\acl(A)$. First note that $G_2^0 \subset G_1^0 \cap G_2$ since $G_2$ is a subgroup  of $G_1$. Let $g$ be as in the assumption of $(iii)$ and take $g_1 \equiv_A g$ with $g_1 \ind_{A} g$. Since $G_2$ is definable over $A$, the type $\tp(g/A)$ determines a coset of the connected component $G_2^0$. In particular, the product $g_1^{-1}\cdot g$ lies in $G_2^0$. By the independence $g_1 \ind_A g$ and genericity of $g$ in $G_1$ we also get that $g_1^{-1}\cdot g$ is  generic in $G_1$. Hence, the group $G_2^0$ contains a generic element of $G_1$ which implies that $G_1^0 \subset G_2^0$. 
 	\end{proof}
	We use the remark as follows:
 	\begin{cor}\label{RemarkExistenceOfGenerics} Let $G$ and $H$ be as in Definition \ref{DefinitionWagnerGenerics}. If the definable hull $\dc_G(H)$ is principal in $G$, then for every $A \supset H$ there exists a type $p $ in $S_1^H(A)$ which is generic in $G$.
 	\end{cor}
 \begin{proof}
 	Choose $p$ to be a generic type for $H$ over $A$, then $p$ is generic in $\dc_G(H)$ by Fact \ref{FactWagnerResults} (\ref{FactWagnerGenericsResult}) and thus in $G$ by Remark \ref{RemarkVariantsGenericity}.
 \end{proof}

\subsection{Back to our Setting}\label{SectionHullSetting}

We return to the assumptions of Section \ref{SectionStationarity} and fix a sufficiently saturated model $\M_0$ of the stable $\lingua_0$-theory $T_0$ which eliminates imaginaries and quantifiers. As in Section \ref{SectionHullGeneralResults}, we assume that $\M_0$ is the universe of a group $(G,\cdot)$. In addition, we now consider  a simple $\lingua$-theory $T$ with $\lingua_0 \subset \lingua$ such that $T_0^\forall \subset T$ (see Remark \ref{RemarkBasicFactsOnSetting}).

Inside $\M_0$, we fix a sufficiently saturated model $\M$ of $T$ and work either in $\M$ or $\M_0$. Furthermore, let $\N$ be a small elementary substructure of $\M$ and $H$ an $\lingua(N)$-definable subset of $M$. 
We now suppose that $(H,\cdot)$ is a group in $\M$ (i.e. closed under $\cdot$ and inverses). 

In particular, both the $N$-points $(H(N),\cdot)$ and the $M$-points $(H(M),\cdot)$ form
(possibly non-definable) subgroups of $(G, \cdot)$.

 	\begin{remark}\label{RemarkDefinableHullSame} In the just described setting, the definable hulls of $H(M)$ and $H(N)$ in $\M_0$ coincide, that is $\dc_G(H(M))=\dc_G(H(N))$. 
\begin{proof}
	 Since  $H(N) \subset H(M)$, we have $\dc_G(H(N)) \subset \dc_G(H(M))$ as type-definable sets.	 
	 
	 On the other hand, if $\psi(x)$ is an $\lingua_0(N)$-formula such that $H(N) \subset \psi(M_0)$, then (by quantifier elimination of $T_0$) \(\mathcal{N} \models \forall x(H(x) \to \psi(x))\). Since $\N \prec \M$ we thus have that $H(M) \subset \psi(M_0)$, so $\dc_G(H(M)) \subset \dc_G(H(N))$ by Fact \ref{FactWagnerResults} (\ref{FactWagnerProprertiesDC}).
\end{proof}
 	\end{remark}

  \textit{Warning}: Since $(H,\cdot)$ is an $\lingua(N)$-definable group in the simple theory $T$, we now have two different notions of genericity for elements in $H$. A type is said to be ($\lingua$-)generic \textit{in} $H$ if it is generic in the simple theory $T$ and generic \textit{for} $H(M)$ as in  Definition \ref{DefinitionWagnerGenerics} if we view $(H(M),\cdot)$ as a subgroup of $(G,\cdot)$.
  
However, these notions are related if $T$ is stable:
 		\begin{remark}\label{RemarkGenericityStableCorrespondance}
 			Let $(G,\cdot)$ and $(H,\cdot)$ be as above and suppose that $T$ is stable. An element $g$ of $H(M)$ is generic for $H(M)$ (working in $\M_0$) if it is $\lingua$-generic in $H$ over $N$ (working in $\M$).
 			
 			In particular, such an element $g$ is $\lingua_0$-generic over $N$ in $\dc_G(H(M))$ by Fact \ref{FactWagnerResults} (\ref{FactWagnerGenericsResult}).
 		\end{remark} 
 		
 		\begin{proof} First work in $\M$. Since $T$ is stable and $\tp_{\lingua}(g/N)$ is generic in $H$, we find  for every formula $\varphi(x)$ in $\tp_\lingua(g/N)$ elements $g_1, \ldots, g_n$ and $h_1, \ldots, h_n$ of $H(M)$ such that $\M \models  \bigvee_{i=1}^n  \varphi(g_i \cdot h \cdot h_i)$ for every element $h$ of $H(M)$.
 		
 		If $\varphi$ is a quantifier-free $\lingua_0$-formula, we then also have  $\M_0 \models  \bigvee_{i=1}^n  \varphi(g_i \cdot h\cdot  h_i)$ since $\M$ is a substructure of $\M_0$. In particular, this is true for all formulas in $\tp_0(g/N)$ by quantifier elimination of $T_0$. Hence, working in $\M_0$, the type $\tp_0(g/N)$ is generic for $H(M)$.
 		\end{proof}
 		We will now consider the definable hulls in $\M_0$ of two different subsets of $\M$, one being contained in the other.
		
 		\begin{lemma}\label{LemmaEquivalentPrincipalHulls}
 			Given two $\lingua(N)$-definable sets $H_1$ and $H_2$ such that $(H_1,\cdot)$ and $(H_2,\cdot)$ are groups and $(H_2,\cdot) \subset (H_1,\cdot)$, the following two conditions are equivalent:
 		\begin{enumerate}[(i)]
 			\item The definable hull $\dc_G(H_2(N))$ is principal in the definable hull $\dc_G(H_1(N))$.
 			\item There exists an element $h$ of $H_2(M)$ that is $\lingua_0$-generic over $N$ in $\dc_G(H_1(N))$.
 		\end{enumerate}
 	
 	\end{lemma}

 \begin{proof}
 	Suppose that  $\dc_G(H_2(N))$ is principal in the definable hull $\dc_G(H_1(N))$.
 	Working in $\M_0$, there exists  an $\lingua_0$-Type $p$ that is finitely satisfiable in $H_2(N)$ and generic in $\dc_G(H_1(N))$ by Remark \ref{RemarkExistenceOfGenerics} applied to $H_2(N)$ as a subgroup of $\dc_G(H_1(N))$.
 	By quantifier elimination of $T_0$, the $\lingua_0$-type $p$ yields a partial $\lingua$-type over $N$ in $\M$ which is still finitely satisfiable in $H_2(N)$. This means that the partial type $p \cup \{H_2(x)\}$ over $N$ is consistent in  $\M$ and thus realized by an element $h$. Hence, we have found an element as in $(ii)$.
 	
 		On the other hand, an element $h$ as in $(ii)$ lies in $\dc_G(H_2(M))$ which  by Remark \ref{RemarkDefinableHullSame} equals $\dc_G(H_2(N))$. Remark \ref{RemarkVariantsGenericity} yields that $\dc_G(H_2(N))$ is principal in $\dc_G(H_1(N))$.
 \end{proof}

 		\section{Imaginaries}\label{SectionImaginaries}
 
	Recall that an $\lingua$-theory $T$ has \textit{geometric elimination of imaginaries} if every imaginary $\alpha$ is $\lingua$-interalgebraic with a real tuple $\bar{a}$.
	The theory has 	\textit{weak elimination of imaginaries} if additionally $\alpha$ is definable over  $\bar{a}$. If $T$ has weak elimination of imaginaries, then $\acl_{\lingua}^{eq}(A)=\dcl_{\lingua}^{eq}(\acl_{\lingua}(A))$ for every real set $A$, so types over real algebraically closed sets are stationary. By \cite[Proposition 3.4]{CF04Imaginaries}, 
	a stable theory has weak elimination of imaginaries if it has geometric elimination of imaginaries and types over real algebraically closed sets are stationary.

Delon showed in \cite[p. 36]{D88Ideaux} that the theory $SCF_p^\infty$ does not have elimination of imaginaries. Her proof easily generalizes to show that geometric elimination also fails. A similar proof was given by Messmer and Wood \cite[Remark 4.3]{MW95der} to show that elimination of imaginaries fails in the theory $DCF_p$ of differentially closed fields of positive characteristic. Pillay and Vassiliev showed in \cite[Lemma 3.3]{PV04Imaginaries} with a different proof that whenever the stable nfcp theory $T_0$ has elimination of imaginaries, the theory $T=T_0P$ of beautiful pairs of $T_0$ has geometric elimination of imaginaries if and only if no infinite group is definable in $T_0$. In particular, the theory of proper pairs of algebraically closed fields does not have geometric elimination of imaginaries.

In this section, we present a criterion for a theory not to have geometric elimination of imaginaries that extracts the essences of Delon's  proof as well as Messmers's and Wood's proof and  which applies to all examples mentioned above. 

We work in the setting of Lemma \ref{LemmaEquivalentPrincipalHulls} in the previous Subsection \ref{SectionHullSetting}: Let $T_0$ be a stable $\lingua_0$-theory (with elimination of imaginaries and quantifiers) and $\M_0$ a sufficiently saturated model. Consider a (simple) $\lingua$-theory $T$ such that $T_0^\forall \subset T$ and choose a sufficiently saturated model $\M$ of $T$ inside $\M_0$.

We again suppose that there are an $\lingua_0$-definable group $(G,\cdot)$ in $M_0$ as well as $\lingua$-definable subsets $H_2 \subset H_1$ of $M$ such that $(H_1,\cdot)$ and $(H_2,\cdot)$ are groups in $\M$; all $G, H_1$ and $H_2$ being defined over a small elementary substructure $\N$ of $\M$.

The first lemma we present is somewhat technical but contains the common idea of the different applications.

\begin{lemma}\label{LemmaCommonEliminationCriterion}
	Suppose that in the setting just described the following conditions are satisfied:
	\begin{enumerate}[(i)]
		\item\label{ConditionLemma1} For tuples $\aq$ and $\bq$ from $M$ holds $\aq \indL_N \bq \quad  \Rightarrow \quad  \aq \indLn_N \bq$.
		\item\label{ConditionLemma2} There exists an element $h$ of $H_2(M)$ that is $\lingua_0$-generic over $N$ in $\dc_G(H_1(N))$.
		\item\label{ConditionLemma3}  There exists an element $g$ of $H_1(M)$ with $g \indL_N h$ such that $\acl_{\lingua}(g,N) \subset \acl_0(g,N)$ as well as $\acl_\lingua(g\cdot h,N)  \subset \acl_0(g\cdot h, N)$.
	\end{enumerate}
	If the canonical parameter of the coset $g \cdot H_2$ of $H_2$ in $H_1$ is $\lingua$-interalgebraic over $N$ with a real tuple, then the coset is $\lingua(N)$-definable. Moreover, we then have that the right-stabilizer $\Stab^r_{H_1}(\tp_{\lingua}(g/N))$ is contained in $H_2$.
\end{lemma}
Note that if we further assume that $T$ is stable and $H_2 \subset \Stab^r_{H_1}(\tp_{\lingua}(g/N))$,  the conclusion of  Lemma \ref{LemmaCommonEliminationCriterion} that $g\cdot H_2$ is $\lingua(N)$-definable yields  $\Stab^r_{H_1}(\tp_{\lingua}(g/N))=H_2$. By \cite[Lemma 4.4]{PGST} we then have that $H_2$ is connected and that $\tp_{\lingua}(g/N)$ is the generic type of $g \cdot H_2$. 

 \begin{proof}
 	Given $g$ and $h$ as in (\ref{ConditionLemma2}) and (\ref{ConditionLemma3}), let $\alpha$ be the canonical parameter of the coset $g\cdot H_2= g \cdot h \cdot H_2$. Suppose that $\alpha$ is $\lingua$-interalgebraic over $N$ with the real tuple $\bar{a}$.

 	By Condition (\ref{ConditionLemma1}), the $\lingua$-independence implies that $g \indLn_N h$. Working now in $\M_0$, observe that $g$ and $h$ are both in $H_1(M) \subset \dc_G(H_1(M))= \dc_G(H_1(N))$ by Remark \ref{RemarkDefinableHullSame}. Thus, genericity of $h$ in $\dc_G(H_1(N))$ yields that 
 	\[g\cdot h \indLn_N g \,.\tag{$\star$}\]
 	
 		Considering $\M$ again, note that $\alpha$ is $\lingua(N)$-definable over any element of the coset because $H_2$ is $\lingua(N)$-definable. Thus, the tuple $\aq$ is $\lingua$-algebraic both over $N,g$ and $N,g\cdot h$. 
 	
 	 Condition (\ref{ConditionLemma3}) and the independence in $(\star)$ now yield that
 	\[\bar{a} \in \acl_0(g,N)\cap \acl_0(g\cdot h,N) \cap M\subset \acl_0(N) \cap M=N.\]
 	Hence, the imaginary $\alpha$ lies in $\acl_{\lingua}^{eq}(N)=N^{eq}$ (since $N$ is a model of $T$). In particular, the coset $g \cdot H_2$ is $N$-definable.\\
 	
 For the moreover part note that it is enough to check that $\St^r_{H_1}(\tp_{\lingua}(g/N))\subset H_2$ where $\St^r_{H_1}(\tp_{\lingua}(g/N))=\{h' \in H_1 \mid \exists \, g'\eqL_N g \text{ with } g' \indL_N h' \text{ and } g' \cdot h'\eqL_N g\}$. 
 Let thus $h'$ be in $\St^r_{H_1}(\tp_{\lingua}(g/N))$, witnessed by $g' \eqL_Ng$. Since $g \cdot H_2$ is $\lingua(N)$-definable, we have that $g' \cdot H_2 =g \cdot H_2=g' \cdot h'  \cdot H_2$. By left-cancellation this implies that $h'$ is in $H_2$.
 \end{proof}
	\begin{remark}\label{RemarkEliminationCriterionChecking}
	Some of the conditions in Lemma \ref{LemmaCommonEliminationCriterion} have already appeared in a previous context in this paper:
	\begin{itemize}
		\item As a model of $T$, the set $N$ is $\lingua$-algebraically closed. Thus Property \ref{HypothesisDcl} introduced in Section \ref{SectionStationarity} yields by Fact \ref{FactIndependenceImplication} that Condition (\ref{ConditionLemma1}) holds. The Property \ref{HypothesisDcl} will be satisfied in all our examples.
		\item Since $T$ is controlled by $T_0$, the second part of Condition (\ref{ConditionLemma3}) is equivalent by Remark \ref{RemarkBasicFactsOnSetting} (\ref{ItemZusammenhangDerAcls}) to   ${\acl_{\lingua}(g,N) = \acl_0(g,N) \cap M}$ and ${\acl_{\lingua}(g\cdot h,N) = \acl_0(g \cdot h,N) \cap M}$.
	\end{itemize}
\end{remark}
In order to show the failure of elimination of imaginaries, 
we first give a criterion for the case that $T$ is stable:

\begin{prop}\label{PropEliminationCriterionStable}
	Consider a stable $\lingua$-theory $T$ with $T_0^\forall \subset T$ for the stable $\lingua_0$-theory $T_0$ (see Section \ref{SectionStationarity}). Inside a sufficiently saturated model $\M_0$ of $T_0$ choose $\M\models T$ also sufficiently saturated.
	
	Suppose that there is an $\lingua_0$-definable group $(G,\cdot)$ in $M_0$ and  an $\lingua$-definable subset $H_1$ of $M$ such that $(H_1,\cdot)$ is a group in $\M$; both $G$ and $H_1$ being defined over a small elementary substructure $\N$ of $\M$. 
	Assume further that:
	\begin{enumerate}[(1)]
		\item\label{ConditionEliminationCriterionStableIndependence} For tuples $\aq$ and $\bq$ from $M$ holds $\aq \indL_N \bq \quad  \Rightarrow \quad  \aq \indLn_N \bq$.
		\item \label{ConditionEliminationCriterionStableTypeAcl} There exists a type $p$ over $N$ of an element of $H_1$ such that for every realization  $g$ of $p$ holds $\acl_{\lingua}(g,N) \subset \acl_0(g,N)$.
		\item\label{ConditionEliminationCriterionStableSubgroup} The right stabilizer $\Stab^r_{H_1}(p)$  contains a proper (infinite) $\lingua(N)$-definable subgroup $H_2$ such that the definable hull $\dc_G(H_2(N))$ is principal in the definable hull $\dc_G(H_1(N))$ (c.f. {Defi\-nition} \ref{DefinitionPrincipal}), that is $\dc_G(H_2(N))^0=\dc_G(H_1(N))^0$.
	\end{enumerate}
	Then, the theory $T$ does not have geometric elimination of imaginaries, even after adding constants for the model $\N$ to the language.
\end{prop}
\begin{proof}
	By Lemma \ref{LemmaEquivalentPrincipalHulls}, the second part of Condition (\ref{ConditionEliminationCriterionStableSubgroup}) yields an element $h$ of $H_2(M)$ that is $\lingua_0$-generic over $N$ in $\dc_G(H_1(N))$. Let $g$ in $H_1$ be a realization of $p$ with $g \indL_N h$. Since $T$ is stable and $H_2 \subset \Stab^r_{H_1}(p)$, it follows that $g \cdot h$ is a realization of $p$ and in particular that $\acl_{\lingua}(g\cdot h,N) \subset \acl_0(g \cdot h,N)$. Hence, all conditions of Lemma \ref{LemmaCommonEliminationCriterion} are satisfied. 
	
	Suppose for a contradiction that the canonical parameter of the coset $g \cdot H_2$ is interalgebraic over $N$ with a real tuple. Lemma \ref{LemmaCommonEliminationCriterion} then yields that  $\Stab^r_{H_1}(g/N)=\Stab^r_{H_1}(p)$ is contained in $H_2$, which contradicts Condition (\ref{ConditionEliminationCriterionStableSubgroup}). 
\end{proof}
	If the group $H_1$ in Proposition \ref{PropEliminationCriterionStable} is connected and $p$ a generic type of $H_1$, the subgroup $H_2$ in Condition (\ref{ConditionEliminationCriterionStableSubgroup}) is simply an infinite index subgroup of  $H_1$. 
	This observation yields the following criterion for simple theories.

	\begin{prop}\label{PropEliminationCriterionSimple}
		Consider a simple $\lingua$-theory $T$ with $T_0^\forall \subset T$ for the stable $\lingua_0$-theory $T_0$ (see Section \ref{SectionStationarity}). Inside a sufficiently saturated model $\M_0$ of $T_0$ choose $\M\models T$ also sufficiently saturated.
		
		Suppose that there is an $\lingua_0$-definable group $(G,\cdot)$ in $M_0$ and  an $\lingua$-definable subset $H_1$ of $M$ such that $(H_1,\cdot)$ is a group in $\M$; both $G$ and $H_1$ being defined over a small elementary substructure $\N$ of $\M$. 
		Assume further that:
		\begin{enumerate}[(1)]
			\item\label{ConditionEliminationCriterionIndependence} For tuples $\aq$ and $\bq$ from $M$ holds $\aq \indL_N \bq \quad  \Rightarrow \quad  \aq \indLn_N \bq$.
			\item[(2')]\label{ConditionEliminationCriterionAcl}For every $\lingua$-generic $g$ in $G(M)$ over $N$ we have that $\acl_{\lingua}(g,N) \subset \acl_0(g,N)$.  
			\setcounter{enumi}{3}
			\item[(3')] \label{ConditionEliminationCriterionIndex} There is an infinite $\lingua(N)$-definable subgroup $H_2$  of $H_1$ of infinite index such that the definable hull $\dc_G(H_2(N))$ is principal in the definable hull $\dc_G(H_1(N))$ (c.f. {Defi\-nition} \ref{DefinitionPrincipal}), that is $\dc_G(H_2(N))^0=\dc_G(H_1(N))^0$.
		\end{enumerate}
		Then, the theory $T$ does not have geometric elimination of imaginaries, even after adding constants for the model $\N$ to the language.
	\end{prop}

	\begin{proof}

		By Lemma \ref{LemmaEquivalentPrincipalHulls} there exists an element $h$ of $H_2(M)$ that is $\lingua_0$-generic over $N$ in $\dc_G(H_1(N))$.
		Work in $\M$ and choose an element $g$ of $H_1$ that is $\lingua$-generic over $Nh$. In particular, $g \indL_N h$ and thus $g \cdot h$ is also generic in $H_1$ over $N$. Condition (\hyperref[ConditionEliminationCriterionAcl]{2'}) yields that $\acl_{\lingua}(g,N) \subset \acl_0(g,N)$ as well as $\acl_\lingua(g\cdot h,N)  \subset \acl_{0}(g\cdot h, N)$
		
		We will again show that the canonical parameter of the coset $g\cdot H_2$ in $H_1$ cannot be geometrically eliminated: Otherwise, by Lemma \ref{LemmaCommonEliminationCriterion}, we have that $\Stab^r_{H_1}(g/N) \subset H_2$. As $g$ is $\lingua$-generic in $H_1$, this stabilizer equals the connected component $H_1^0$ which contradicts that $H_2$ has infinite index in $H_1$.
	\end{proof}

\begin{remark}
	Most parts of the proofs of the two previous criteria were carried out in the stable theory $T_0$. The stability or simplicity of $T$ was mainly used to choose an appropriate element $g$. 
	Indeed, suppose that $T$ is an arbitrary $\lingua$-theory controlled by $T_0$ such that Condition (\hyperref[ConditionEliminationCriterionIndex]{3'}) of Proposition \ref{PropEliminationCriterionSimple} holds. Assuming moreover that $\acl_{\lingua}(g,N) \subset \acl_0(g,N)$ for every $g$ in $H_1$ and that there exists an independence relation $\indS$ in $T$ satisfying full existence with $\indS \Rightarrow \indLn$ over $N$ (instead of  (\ref{ConditionEliminationCriterionIndependence})), the proof of Lemma \ref{LemmaCommonEliminationCriterion} yields that all cosets of $ H_2$ are $N$-definable which is a contradiction by a straightforward compactness argument.   
	
	In particular, the criterion applies to some non-simple theories; for example d'Elbée's $NSOP_1$-theory $ACFG_p$ presented  in \cite[p.16]{DE19additive}, the model companion of the theory of algebraically closed fields in positive characteristic with a predicate for an additive subgroup. 
\end{remark}

		In the proofs of Proposition \ref{PropEliminationCriterionStable} and \ref{PropEliminationCriterionSimple} we only used that $\dc_G(H_2(N))^0=\dc_G(H_1(N))^0$ to find an element $h$ in $H_2$ such that for $g$ in $H_1$ (generic or realizing the type $p$) with $g \indL_N h$ holds $g \indLn_N g\cdot h$.
		 We will now see that in the case of $g$ being generic in $H_1$ and $T$ stable, the above gives an equivalent condition. In order to do so, we need to recall a result of Ziegler, originally stated for abelian groups in \cite{Z06Cosets}. For the purpose of this note, we will use a generalized reformulation given by Blossier, Martin-Pizarro and Wagner in \cite{BMPW15Tore}.

\begin{fact}(c.f. \cite[Lemme 1.2]{BMPW15Tore})\label{FactZieglersLemma}
	Let $(G,\cdot )$ be a type-definable group  in a stable theory and  suppose there are elements $g$ and $h$ of $G$ such that $g, h$ and $g\cdot h$ are pairwise independent over  a set $A$. Then $h$ is generic in the coset $\Stab_G(\stp(h/A)) \cdot h$ and moreover ${g\cdot \Stab_G(\stp(h/A))\cdot g^{-1}}=\Stab_G(\stp(g/A))$.
\end{fact}

	\begin{lemma}
		Suppose that $T$, $T_0$, $(G,\cdot)$ and $(H_1,\cdot)$ are as in Proposition \ref{PropEliminationCriterionStable} and let $(H_2,\cdot)$ be an $\lingua(N)$-definable subgroup of $H_1$. Moreover suppose that Property \ref{HypothesisDcl} (c.f. Section \ref{SectionStationarity}) holds. The following are equivalent:
		\begin{enumerate}[(i)]
			\item The definable hull $\dc_G(H_2(N))$ is principal in the definable hull $\dc_G(H_1(N))$ (see Condition (\hyperref[ConditionEliminationCriterionIndex]{3'}) of Proposition \ref{PropEliminationCriterionStable}).
			\item There exist elements $h$ in $H_2$ and $g$ in $H_1$ such that		
				$g$ is $\lingua$-generic in $H_1$ over $N$ with $g \indL_N h$ and $g \indLn_N g\cdot h$. 
		\end{enumerate}

		\end{lemma}
	\begin{proof}
	First note that Property \ref{HypothesisDcl} yields Condition $(\ref{ConditionEliminationCriterionStableIndependence})$ of Proposition \ref{PropEliminationCriterionStable} (by Fact \ref{FactIndependenceImplication}).
	The proof of Lemma \ref{LemmaCommonEliminationCriterion} (using   Lemma \ref{LemmaEquivalentPrincipalHulls}) then shows that $(i)$ implies $(ii)$. 
	
	On the other hand suppose that $(ii)$ holds.
	Let $h$ in $H_2$ and $g$ in $H_1$ generic over $N,h$ be as in the assumption. Since  $g \indL_N h$ , we obtain $g\cdot h \indL_N h$ by genericity of $g$. As Condition  $(\ref{ConditionEliminationCriterionStableIndependence})$ holds, the $\lingua$-independences imply that \[g \indLn_N h \quad \text{ and } \quad g\cdot h \indLn_N h.\]
		From our assumption, it follows that $g, h$ and $g\cdot h$ are pairwise $\lingua_0$-independent elements of $\dc_G(H_1(M))=: \tilde{G}$ (for simplicity of notation).

		 Recall that all $\lingua_0$-types are stationary over $N$ by Fact \ref{FactStationarityT0}. By Fact \ref{FactZieglersLemma}, the type $\tp_0(h/N)$ is hence generic in a coset of its stabilizer and the latter is a conjugate of $\Stab_{\tilde{G}}(\tp_0(g/N))$. By Remark \ref{RemarkGenericityStableCorrespondance}, the element $g$ is $\lingua_0$-generic over $N$ in $\tilde{G}$ and thus $\Stab_{\tilde{G}}(\tp_0(g/N))=\tilde{G}^0$. Since the connected component is a normal subgroup of $\tilde{G}$, we have $\Stab_{\tilde{G}}(\tp_0(h/N)) =\tilde{G}^0$. 
		
		In particular, the element $h$ of $H_2(M)$ 
		is $\lingua_0$-generic over $N$ in $\tilde{G}=\dc_G(H_1(M))$. By Lemma \ref{LemmaEquivalentPrincipalHulls} this is equivalent to (i) (using that $\dc_G(H_1(M))=\dc_G(H_1(N))$ by Remark \ref{RemarkDefinableHullSame}).
	\end{proof}

	\subsection{Examples}
	Since Lemma \ref{LemmaCommonEliminationCriterion} was inspired by Delon's proof for separably closed fields of infinite degree of imperfection, the criterion in Proposition \ref{PropEliminationCriterionStable} or \ref{PropEliminationCriterionSimple} can easily be applied to this theory. Moreover, these propositions yield a new proof of non-elimination for beautiful pairs of theories with a definable groups.

	\begin{cor}\label{CorNonElSCFandPairs}
		The following theories do not have geometric elimination of imaginaries even after adding constants for a small submodel to the language:
		\begin{enumerate}[a)]
			\item The theory $SCF_p^\infty$ of separably closed fields of infinite degree of imperfection.
			\item The theory $T_0P$ of beautiful pairs of a stable $\lingua_0$-theory $T_0$ (with elimination of quantifiers and imaginaries) where $T_0$ does not have the finite cover property and an infinite group is definable in $T_0$. 
		\end{enumerate}
	\end{cor}
	In particular the theories $ACF_pP$, $DCF_0P$ and $SCF_p^eP$ of beautiful pairs of algebraically closed fields, differential closed fields in characteristic $0$ resp. separably closed fields of finite degree of imperfection do not have geometric elimination of imaginaries; yet types over  real algebraically closed sets are stationary.

	\begin{proof}

	The desired statement follows from Proposition \ref{PropEliminationCriterionSimple} after checking that all conditions hold. (We could also use Proposition \ref{PropEliminationCriterionStable}, but since our group $H_1$ is connected in both cases, this amounts to checking the same conditions.) \\
		a) 		As in Section \ref{SectionExamples}, we consider $T=SCF_p^\infty$ in the language $\lingua=\{0,1,+,-,\cdot , ^{-1}\}$ and $T_0=ACF_p$ with $\lingua_0=\lingua$. Given models $\M_0\models T_0$ and $\M \models T$ as in Proposition \ref{PropEliminationCriterionSimple}, we put $G=(M_0,+)$ the additive group of the algebraically closed field, and $H_1=(M,+)$ the additive group of $M$ itself. 
	
		\begin{enumerate}[(1):]

			\item This follows immediately from the characterization of independence given by Srour in \cite{S86Independence} (see Section \ref{SectionExamples}).
			\item[(2'):] Note that an element of $M$ is additively generic over $N$ if it does not lie in the field $M^p \cdot N$. In particular, the extension $M/N(g)$ of fields is separable which by the characterization of algebraic closure in separably closed fields yields Condition (\hyperref[ConditionEliminationCriterionAcl]{2'}).
			\item[(3'):] Set $H_2=M^P$ the subfield of $p$-powers which is by definition a subgroup of infinite index. Since $ACF_p$ is strongly minimal and $H_2(N)$ and $H_1(N)$ are infinite subgroups of $(M_0,+)$, it follows that $\dc_G(H_2(N))=(M_0,+)=\dc_G(H_1(N))$. 
		\end{enumerate}
		b) Let $T_0$ be a stable nfcp theory and consider the theory $T=T_0P$ of beautiful pairs of $T_0$ as in Section \ref{SectionExamples} in the language $\lingua=\lingua_0 \cup \{P\}$. Choose $\M \models T$ sufficiently saturated, then by definition $\M\models T_0$ (so we can put $\M_0=\M$). By assumption, there exist an $\lingua_0$-definable group $(G,\cdot)$ in $T_0$ which we may assume to be defined over $P(M)$ by saturation of the predicate. We choose a small elementary substructure $\N$ containing the necessary parameters and will now show that $T$ does not eliminate imaginaries in $\lingua(N)$. For simplicity we assume again that $G = M_0$. Put $H_1=G$ (since $M_0=M$). 
	
		\begin{enumerate}[(1):]
			
				\item This is immediate from the characterization of independence in \cite{BYPV03Pairs} (see Section \ref{SectionExamples}).
				\item[(2'):] 	By the characterization of independence, the element $g$ of $H_1$ is generic in $T=T_0P$ over $N$ if and only if it is generic in $T_0$ over $N,P(M)$. 
				In particular, for $g$ generic in $H_1$ over $N$ holds	 \[g \indLn_{P(N)} N,P(M)\]	 
				which implies that the substructure $\langle g,N \rangle_{\lingua}$ is $P$-independent. This yields Condition (\hyperref[ConditionEliminationCriterionAcl]{2'}) by the characterization of the algebraic closure in \cite[Lemma 2.6]{PV04Imaginaries}.
				\item[(3'):] Let $H_2=G(M)\cap P(M)$ be the $\lingua(N)$-definable subgroup of $P$-points of $G$. Note that the index is infinite by saturation. 
				Moreover, since $H_1(M)=G$, we immediately have $\dc_G(H_1(N))=\dc_G(H_1(M))=G$ by Remark \ref{RemarkDefinableHullSame}.

				Similarly $\dc_G(H_2(N))=\dc_G(H_2(M))$. Now, the subgroup $H_2(M)$ is a subset of $P(M)$. By Fact \ref{FactWagnerResults} (\ref{FactWagnerProprertiesDC}), we have that $\dc_G(H_2(M))$ is given by the $M$-points of $\dc_{G(P(M))}(H_2(M))$ and hence equals $G$ as $H_2(M)=G(P(M))$. Thus Condition (\hyperref[ConditionEliminationCriterionIndex]{3'}) holds.
		\end{enumerate}
	\vspace{-6mm}
	\end{proof}
Note that the result for $SCF_p^\infty$ answers negatively Yoneda's question in  \cite[Question 3.1 (4)]{Y24Eliminaiton} whether $SCF_{p}^\infty$ has geometric elimination of imaginaries in the language of fields. 

We will now present an application of Proposition \ref{PropEliminationCriterionStable} with a non-generic type $p$ in Condition (\ref{ConditionEliminationCriterionStableTypeAcl}): A straightforward adaption of the proof of Messmer and Wood \cite[Remark 4.3]{MW95der} yields the failure of non-elimination in a broader class of differential fields.

A differential field $K$ of positive characteristic $p$ is separably differentially closed if $K$ is existentially closed in all separable differential extensions. The \emph{differential degree of imperfection} is the cardinality of a $p$-basis of $C_K$ over $K^p$. Ino and León Sanchez introduced the theory $SDCF_{p,e}$ of \textit{separably differentially closed fields} of fixed differential degree of imperfection $e$ in $ \setN\cup\{\infty\}$ in \cite{LM24indep} and showed completeness and stability. The case $e=0$ yields the theory $DCF_p$ which was already studied by Wood in \cite{Wood1, Wood2, Wood3} and Shelah in \cite{ShelahDCF} as the characteristic $p$ analogue of $DCF_0$.

 We consider $T=SDCF_{p,e}$ in the language $\lingua=\{0,1,+,-,\cdot,^{-1},D\}$ where $D$ is a unary function symbol interpreted as the derivation.  Note that for every $e$, the theory $T$ contains the universal part of the theory $T_0=ACF_p$ in the language $\lingua_0$ of fields.

\begin{fact}\label{FactAclInSDCF} 
	Let $K\models T$ and let $k_0$ be a differential subfield such that the extension of fields $K/k_0$ is separable, then $\acl_{\lingua}(k_0) =\acl_0(k_0) \cap K$.
\end{fact}
	For $e=\infty$ this follows directly from \cite[Lemma 4.8]{LM24indep} (using \cite[Lemma 3.8]{LM24indep}). The same works for $e$ finite since $T$ can be viewed by the quantifier elimination result in \cite[Section 6.1]{IS23differentially} as the model completion of an almost derivation-like theory (c.f. \cite[Remark 3.10]{LM24indep}).

\begin{lemma}
	For every $e$ in $\setN\cup\{\infty\}$, the theory $SDCF_{p,e}$ does not have geometric elimination of imaginaries even after adding constants for a small submodel to the language.
\end{lemma}

\begin{proof}
	As above let $T$ be the theory $SDCF_{p,e}$ and $T_0=ACF_p$. Given models ${\M_0\models T_0}$ and $\M \models T$ as in Proposition \ref{PropEliminationCriterionStable}, we put $G=({M_0\setminus \{0\}},\cdot)$ the multiplicative group of the algebraically closed field and $H_1=({M\setminus \{0\}},\cdot)$ the multiplicative group of $M$. We are done by Proposition \ref{PropEliminationCriterionStable} after checking that the conditions hold.
	
	\begin{enumerate}[(1):]
		\item This condition follows from Fact \ref{FactIndependenceImplication} since Property \ref{HypothesisDcl} is satisfied in every theory of fields (see Remark \ref{RemarkDClRedukt}).
		\item Let $p$ be the type  of an element $g$ with $D(g)=g$ and such that $g$ does not lie in the field $C_M \cdot N$ (i.e. $g$ is $p$-independent over $N$). Note that these conditions define a complete type by the quantifier elimination results of Ino and León Sánchez \cite[Theorem 6.3 and 6.6]{IS23differentially}.
		For any realization $g$ of $p$, the field $N(g)$ is by definition a differential subfield of $M$ with $M/N(g)$ separable. We hence obtain Condition (\ref{ConditionEliminationCriterionStableTypeAcl}) from Fact \ref{FactAclInSDCF}.
		\item An easy computation shows that $\Stab_{H_1}^r(p)=C_M\setminus \{0\}$. We thus put $H_2=M^{p^2}\setminus \{0\}$ which is for all $e$ a proper subgroup (since $M^p \subset C_M$). As in Corollary \ref{CorNonElSCFandPairs} we have $\dc_G(H_2(N))=({M_0\setminus \{0\}},\cdot )=\dc_G(H_1(N))$ by strongly minimality of $ACF_p$.
	\end{enumerate}
\vspace{-6mm}
\end{proof}
	For the case $e=\infty$ we could have alternatively done the same proof as for $SCF_p^\infty$ (but with $H_1=C_M$ instead of $M$).

\begin{remark}
	In fact, our proof of stationarity over real algebraically closed sets from Theorem \ref{TheoremMain} also applies to $SDCF_{p,\infty}$. Similar to the case of $SCF_p^\infty$, an $\lingua$-substructure $A$ of $\M$ is \emph{strong} if the field extension $M/A$ is separable. Using the characterization of independence given by León Sánchez and Mohamed in \cite{LM24indep} with respect to independence in $SCF_p^\infty$, the other properties also follow.
	
	Actually, an inspection of the proof of 
	\cite[Theorem 2.14]{LM24indep} 
	shows that stationarity transfers from $SCF_p^\infty$ to $SDCF_{p,\infty}$ over arbitrary $\lingua$-algebraically closed sets (and not just models). Hence, types over real algebraically closed sets in $SDCF_{p,\infty}$ are stationary. 
\end{remark}

\end{document}